\documentclass[a4paper,leqno]{amsart}

\usepackage{amsmath}
\usepackage{amsthm}
\usepackage{amssymb}
\usepackage{amsfonts}
\usepackage[applemac]{inputenc}
\usepackage{mathrsfs}
\usepackage{color}

\def\R{{\mathbb R}}
\def\N{{\mathbb N}}

\def\virgp{\raise 2pt\hbox{,}}

\def\({\left(}
\def\){\right)}
\def\<{\left\langle}
\def\>{\right\rangle}
\def\le{\leqslant}
\def\ge{\geqslant}

\def\d{{\partial}}

\def\eps{\varepsilon}

\renewcommand{\d}{\partial}

\newcommand{\modu}[1]{\left|#1\right|}

\newcommand{\pare}[1]{\left(#1\right)}

\theoremstyle{plain}
\newtheorem{theorem}{Theorem}[section]
\newtheorem{lemma}[theorem]{Lemma}
\newtheorem{corollary}[theorem]{Corollary}

\theoremstyle{definition}

\newtheorem{remark}[theorem]{Remark}
\newtheorem*{remark*}{Remark}

\numberwithin{equation}{section}

\begin{document}

\title[Dispersion managed NLS]
{Well-Posedness and averaging of NLS with time-periodic dispersion management}
\author[P. Antonelli]{Paolo Antonelli}
\address[P. Antonelli]{Department of Applied Mathematics and
Theoretical Physics\\
CMS, Wilberforce Road\\ Cambridge CB3 0WA\\ England.}
\email{p.antonelli@damtp.cam.ac.uk}
\author[J.-C. Saut]{Jean-Claude Saut}
\address[J.-C. Saut]{Laboratoire de Math\'ematiques, Universit\'e Pari-Sud, Batiment 425, 91405 Orsay Cedex, France.}
\email{jean-claude.saut@math.u-psud.fr}
\author[C.~Sparber]{Christof Sparber}
\address[C.~Sparber]
{Department of Mathematics, Statistics, and Computer Science, M/C 249, University of Illinois at Chicago, 851 S. Morgan Street, Chicago, IL 60607, USA.}
\email{sparber@math.uic.edu}

\begin{abstract}
We consider the Cauchy problem for dispersion managed nonlinear Schr\"odinger equations, where the dispersion map is assumed to be periodic and piecewise constant in time. We 
establish local and global well-posedness results and the possibility of finite time blow-up. In addition, we shall study the scaling limit of fast dispersion management and establish 
convergence to an effective model with averaged dispersion.
\end{abstract}

\date{\today}


\thanks{P. Antonelli is supported by Award No. KUK-I1-007-43, funded by the King Abdullah University of Science and Technology (KAUST). C.~Sparber 
acknowledges support of the National Science Foundation through grant no.\ DMS-1161580}
\maketitle

\section{Introduction}\label{sect:intro}

In this work, we study the Cauchy problem for the following class of {\it dispersion managed nonlinear Schr\"odinger equations} (NLS):
\begin{equation}\label{eq:disp_nls}
i\d_tu+\gamma(t)\Delta u+|u|^{p-1}u=0\qquad u(t_0,x)=\varphi(x).
\end{equation}
for $(t,x) \in \R_+\times \R^d$, a given $t_0 \in \R_+$, and $p> 1$ to be specified below. 
In addition, we assume that the {\it dispersion map} $\gamma$ is $1-${\it periodic}, i.e. $\gamma(t+1)=\gamma(t)$, and piecewise constant:
\begin{equation}\label{eq:gamma}
\gamma(t)=
\left\{\begin{array}{l}
 \gamma_+ \ \quad  \ 0< t \le t_+ \, ,\\
-\gamma_- \quad  t_+ < t \le 1 \, ,
\end{array}\right.
\end{equation}
where $\gamma_\pm >0$ are some positive constants and $t_+ \in (0,1)$. 

Our main motivation stems from models in nonlinear fiber optics. Indeed, the case of a cubic NLS, i.e. $p=3$, in $d=1$ spatial dimension 
naturally arises as an envelope equation for electromagnetic wave propagation in optical fibers exhibiting a (weak) Kerr nonlinearity. 
In this context, the variable $t\in \R_+$ actually corresponds to the distance along the fiber, and $x\in \R$ denotes
the (retarded) time. The coefficient $\gamma(t)$ consequently models a periodically varying dispersion along the fiber, cf. \cite{Ag, SuSu}. 
Dispersion managed NLS in $d=2$ spatial dimensions are also physically relevant, see, e.g., \cite{ABS}. 
The technique of dispersion management was invented to balance the effects of nonlinearity 
and dispersion in such a way that stable nonlinear pulses (solitary waves) are supported over long distances, cf. \cite{Ku, Lu1, Lu2, ZGJT}. Due to the enormous practical implications, 
there is a huge literature concerned with the qualitative properties of \eqref{eq:disp_nls}. Most often, however, the results are based on non-rigorous asymptotics and/or numerical simulations. 
A notable exception is the regime of so-called strong dispersion management where several rigorous results are available for the corresponding asymptotic model, see the discussion below.

In contrast to that, we shall work directly on the dispersion managed NLS and in the following prove several rigorous results 
concerning the well-posedness of \eqref{eq:disp_nls} and its asymptotic behavior in the case of rapidly varying dispersion. Having in mind 
the physics background of fiber optics, we shall focus on the well-posedness theory for (large) data $\varphi\in L^2(\R^d)$ or $\varphi \in H^1(\R^d)$. 
By multiplying \eqref{eq:disp_nls} with $\overline u$, integrating w.r.t. $x\in \R^d$, and taking the imaginary part of the resulting identity, we (formally) obtain {\it mass conservation}:
\[
\| u(t,\cdot) \|_{L^2(\R^d)} = \| \varphi \|_{L^2(\R^d)}.
\]
Note that in the case {\it without} dispersion management, i.e. $\gamma(t) = \gamma \in \R$, equation \eqref{eq:disp_nls} also conserves the {\it energy}
\begin{equation}\label{eq:energy}
E(t)= \frac{\gamma}{2} \int_{\R^d} | \nabla u(t,x) |^2 dx -  \frac{1}{p+1} \int_{\R^d}  | u(t,x)|^{p+1} dx = E(0).
\end{equation}
Thus, in the {\it defocusing case} $\gamma<0$ and when $p<\frac{d+2}{d-2}$, one immediately infers a uniform bound on the $H^1$ norm of $u$, prohibiting the appearance of {\it finite-time blow-up}. 
In view of \eqref{eq:gamma}, we can think of \eqref{eq:disp_nls} as switching in between focusing and defocusing behavior (with respective dispersion coefficient $\gamma_+>0$ and $-\gamma_-<0$) and we 
consequently expect the appearance of finite time blow-up. That this is indeed the case will be proved in Section \ref{sec:blow}. 
\begin{remark}
Note that this switching between focusing and defocusing behavior makes our problem very different from the NLS type models studied in \cite{Fa, La}. In \cite{Fa} the author considers an equation 
of the form \eqref{eq:disp_nls} with time-dependent dispersion $\gamma(t)\ge 0$, possibly vanishing (with finite order) on a discrete set of points, a typical example being 
$\gamma(t) = |t - t_0|^\lambda$ for some $\lambda >0$. In \cite{La} an NLS type model with time-dependent coefficients but with an additional constant coefficient third order 
spatial derivative is studied. The behavior of this model is similar to the KdV equation.
\end{remark}
Even if finite time blow-up in general can not be ruled out, one may still wonder if sufficiently {\it fast} switching between the focus and defocusing step can at least delay, the appearance of blow-up. In order 
to gain more insight we shall therefore study the scaling limit corresponding to {\it fast dispersion management}, i.e. we shall consider 
\begin{equation}\label{eq:fdisp_nls}
i\d_tu_\eps+\gamma\left(\frac{t}{\eps}\right)\Delta u_\eps+|u_\eps|^{p-1}u_\eps=0, \qquad u_\eps(t_0,x)=\varphi(x).
\end{equation}
where $0<\eps \ll 1$ denotes a small parameter. This regime has been studied using formal asymptotics and numerical simulations in, e.g., \cite{BK1, BK2, YKT}. 
As $\eps \to 0_+$ we expect the behavior of the solution $u^\eps$ to be close (in some sense to be made precise, see Section \ref{sec:fast}) 
to the solution of the {\it averaged NLS}
\begin{equation}\label{eq:av_nls}
i\d_t u_0+\langle \gamma \rangle \Delta u_0+|u_0|^{p-1}u_0=0, \qquad u_0(t_0,x)=\varphi(x),
\end{equation}
where we denote by
\begin{equation}\label{eq:average}
\langle \gamma \rangle: = \int_0^1 \gamma(\tau) d\tau,
\end{equation}
the average dispersion coefficient (which can be either positive or negative). In the case of mean zero dispersion $\langle \gamma \rangle = 0$ 
this scaling limit provides a possible explanation for the stabilizing effects of dispersion management, see Corollary \ref{cor:zero}. 

The situation above should, however, be distinguished from the case of {\it strong dispersion management}, cf. \cite{Lu1, GT} for some physical motivation. One thereby considers  
a dispersion map of the form $\frac{1}{\eps} \gamma\left(\frac{t}{\eps}\right)$, which, as $\eps \to 0_+$, leads to an effective description by a non-local equation, originally introduced in \cite{GT}. 
This model has been rigorously studied by several authors, see, e.g., \cite{EHL, HL, PZ, St, ZGJT}. In particular, it provides a mathematical basis for the definition of {\it dispersion managed solitons} \cite{EHL}. 

\begin{remark} 
A similar situation is analyzed in \cite{BD, DT}, where the authors consider NLS with fast {\it random} dispersion management. More precisely, they consider 
dispersion maps of the form $\frac{1}{\eps} \gamma \left(\frac{t}{\eps^2}\right)$, where $\gamma$ is is a (smooth) stationary random process, and prove convergence of the equation towards 
an NLS with white noise dispersion. Finally, we want to mention that in \cite{CS, DG} the somewhat dual problem of 
NLS with (rapidly) {\it time-oscillating nonlinearity} has been analyzed (see also \cite{CPS,  PS} for a related studies on the 
KdV equation).
\end{remark}

This paper is organized as follows: In Section \ref{sec:lin} we shall collect some basic properties of the linear equation to be used in the following. Section \ref{sec:well} 
establishes global well-posedness in the $L^2$ sub-critical regime. In addition, we set up a local well-posedness result in $H^1$ theory and prove the existence of 
finite time blow-up. Section \ref{sec:fast} is concerned with the fast dispersion limit. Finally, we shall collect some concluding remarks on possible generalizations 
and closely related problems in Section \ref{sec:remarks}.


\section{Basic properties of the linear equation}\label{sec:lin}

\subsection{The linear propagator} Before studying the nonlinear Cauchy problem we shall collect some basic facts about the associated linear equation
\begin{equation}\label{eq:disp_lin}
i\d_t u_{\rm lin}+\gamma(t)\Delta u_{\rm lin}=0\qquad u_{\rm lin}(t_0,x)=\varphi(x).
\end{equation} 
In the following, we shall denote by 
\[
\Gamma(t,s):= \int_{s}^t \gamma(\tau) d \tau
\]
the {\it cumulative dispersion} on the time-interval $[s,t]\subset \R$ and the associated propagator $U(t,s)$ by
\begin{equation}\label{eq:U}
U(t,s) f(x):= e^{ i \Gamma(t, s) \Delta} f(x) = \int_{\R^d}  e^{- i \Gamma(t, s) |\xi |^2} \widehat f(\xi) e^{ i x\cdot \xi} d\xi,
\end{equation}
where $\widehat f $ denotes the Fourier transform of $f\in L^2(\R^d)$. We directly infer the following result. 

\begin{lemma} \label{lem:group}
Let $t_0\in \R$ be fixed. Then the mapping $t\mapsto U(t,t_0)$ defines a family of strongly continuous unitary operators on $L^2(\R^d)$, such that for all $t\in \R_+$ it holds
\[
u_{\rm lin}(t,x) = U(t,t_0)\varphi(x).
\]
Explicitly, we find
\begin{equation}\label{eq:explicit}
u_{\rm lin}(t,x) = \frac{e^{-i \pi d /4}}{|4\pi  \Gamma(t,t_0)|^{d/2}} \int_{\R^d} e^{i |x -y|^2/ (4 \Gamma(t,t_0))} \varphi(y) dy,
\end{equation}
provided $\Gamma(t,t_0)\not =0$.
\end{lemma}
Some remarks are in order: First, one should note that for all $s,t\in \R$: \[U(t,t_0)U(t_0,s)=U(t,s),\]
and we also have \[U(t,s)^*= U(t,s)^{-1} = U(s,t).\] However, $U(t,s)\not = U(t-s,0)$, since 
\[
\int_{s}^t \gamma(\tau) d \tau = \int_{0}^{t-s} \gamma(\tau+s) d \tau \not =  \int_{0}^{t-s} \gamma(\tau) d \tau,
\]
unless $\gamma(t)=\text{const.}$ {\it for all} $t\in \R$. In other words, $U(t,s)$ is not a group. Second, we note that for $\varphi \in L^2(\R^d)$, it holds
\[
\partial_t u \in BV(\R;  H^{-2}(\R^d)),
\]
in view of \eqref{eq:disp_nls} and the fact that $\gamma(t)$ is piecewise constant (including jump discontinuities). 
In particular, $\partial_t u$ is {\it not} continuous in time, in contrast to the case with constant dispersion.

Next, we recall the decomposition
\[
\gamma(t) = \langle \gamma \rangle + \gamma_0(t),
\]
where $\langle \gamma \rangle \in \R$ denotes the average defined in \eqref{eq:average}, and $\gamma_0(t)$ is $1-$periodic with mean zero. Using this, we can write
\begin{equation}\label{eq:Gamma_decom}
\Gamma(t,s) = \langle \gamma \rangle (t-s) +  \int_{s}^t \gamma_0(\tau) d \tau.
\end{equation}
We consequently infer that if $\langle \gamma \rangle = 0$, then $\Gamma(t,t_0)=0$ for all $t \in \R$ such that $(t-t_0)\in \N$. On the other hand, if $\langle \gamma \rangle \not = 0$, then $\Gamma(t,t_0)$ 
vanishes at most finitely many times, since 
\begin{equation} \label{eq:gamma_bound}
-(\gamma_- + \langle \gamma \rangle) \le \int_{t_0}^t \gamma_0(\tau) d \tau \le \gamma_+ - \langle \gamma \rangle,
\end{equation}
and thus $|\Gamma(t,t_0)| \gtrsim |t|$, provided $t\ge t_0$ is sufficiently large. Also note that if $t_0 \in \N$, i.e. if we start our time-evolution at the beginning of focusing step, then 
\[
\vartheta (t): = \int_{t_0}^t \gamma_0(\tau) d \tau \ge 0,
\]
for all times $t \in \R$, whereas if we start at $t=t_0$ at the beginning of a defocusing step, then $\vartheta (t) \le 0$ for all times.
\begin{remark}
In the particular situation of mean-zero dispersion, i.e. $\langle \gamma \rangle=0$, the solution $u_{\rm lin}$ is found to be $1-$periodic in time. Indeed, in view of \eqref{eq:U}, we infer that the Fourier transformed 
solution is given by
\[
\widehat u_{\rm lin}(t, \xi)=\exp\left(i|\xi|^2 \int_{t_0}^t \gamma_0(\tau) d \tau \right) \widehat \varphi(\xi),
\]
which satisfies $\widehat u_{\rm lin}(n, x)=\widehat \varphi(x)$, for all $n \in\N$. 
\end{remark}

\subsection{Dispersive properties} From what is said above it is clear that $U(t,s)$ in general does {\it not} allow to infer uniform (in time) dispersive estimates (analogous to the usual Schr\"odinger group), since for 
arbitrary $s,t \in \R$
\[
|\Gamma(t,s) |= \Big | \int_{s}^t \gamma(\tau) d \tau \Big | \not  \gtrsim | t-s|,
\]
due to fact that $\gamma(t)$ changes sign. 
\begin{remark}
Clearly, this would not be the case if $|\gamma(t)|>0$ and $1-$periodic, in which case the behavior of \eqref{eq:disp_nls}, for all times $t\in \R$, would be either focusing or defocusing 
(but with time-dependent dispersion coefficient), see, e.g., \cite{Fa}.
\end{remark}

The fact that $\gamma(t)$ is assumed to be piecewise constant allows us to obtain the following result.
\begin{lemma} Let $t_+ \in (0,1)$ be given and $t, s \in  (n , n + t_+]$, or $t,s \in (n+ t_+ , n+1]$, for $n\in \N$. Then, for $t\not = s$, it holds
\[ \|U(t,s) f \|_{L^\infty(\R^d)}\lesssim |t-s|^{- d/2} \|f\|_{L^1(\R^d)}.\]
\end{lemma}
The point is that both $t$ and $s$ have to be in the {\it same} time-interval corresponding to either a focusing or defocusing step.
\begin{proof}
The proof follows directly from the representation formula \eqref{eq:explicit} and the fact that if $t,s \in  (n , n + t_+] $, or $t,s \in (n+ t_+ , n+1]$ we have
\[
\Gamma(t,s) =  \int_{s}^t \gamma(\tau) d \tau = \pm \gamma_\pm (t-s),
\]
in view of \eqref{eq:gamma}.
\end{proof}

Clearly, we have that
\begin{equation}\label{eq:decom}
\bigcup_{n\in \mathbb Z} \big ( (n , n + t_+] \cup (n+ t_+ , n+1] \big )=\R, 
\end{equation}
and thus we can split the time-axis $\R$ into a union of sub-intervals on which $U(t,s)$ allows for the usual dispersive behavior of the Schr\"odinger group. 
As a consequence, we know that the usual Strichartz estimates hold on each such time-interval \cite{GV, KT}. To this end, let us recall that a pair $(q,r)$ is {\it admissible} if $2\le r
  \le\frac{2d}{d-2}$ (resp. $2\le r\le \infty$ if $d=1$, $2\le r<
  \infty$ if $d=2$)  
  and \cite{Caz}:
$$\frac{2}{q}= d\left( \frac{1}{2}-\frac{1}{r}\right).$$

\begin{lemma} \label{lem:strichartz}
Let $n\in \N$ and $t_+ \in (0,1)$ be given. Then, for each admissible $(q,r)$, $(q_1,r_1)$, $(q_2,r_2)$, and each 
time-interval $I_n\subset (n , n + t_+]$, or $I_n \subset (n+ t_+ , n+1]$, we have:\\
$1.$ There exists $C_r=C(r,I_n)$, such that for $t_0\in I_n$ and for any $f \in
L^2(\R^d)$ it holds 
\[
    \left\| U(\cdot, t_0)f \right\|_{L^q(I_n;L^r(\R^d))}\le C_r \|f
    \|_{L^2(\R^d)}.
\]
$2.$ There exists $C_{r_1,r_2}=C(r_1,r_2,I)$, such that for any $F\in
L^{q'_2}(I_n;L^{r'_2})$ it holds
\[
      \Big\| \int_{I_n\cap\{s\le
      t\}} U(t,s)F(s) d s 
      \Big\|_{L^{q_1}(I_n;L^{r_1}(\R^d))}\le C_{r_1,r_2} \left\|
      F\right\|_{L^{q'_2}(I_n;L^{r'_2}(\R^d))} .
\]
\end{lemma}

Again, we can only conclude the existence of Strichartz estimate locally on each time-interval corresponding to, either, a focusing or defocusing step.

\begin{remark} One may want to compare this to the case with {\it smooth} periodic dispersion management. For example, let $\gamma(t) = \cos ( t)$ and $t_0 = 0$. In this case we find
\[
u_{\rm lin}(t,x) = \frac{e^{-i \pi d /4}}{|4 \pi  \sin ( t ) |^{d/2}} \int_{\R^d} e^{i |x -y|^2/ (4 \sin ( t))} \varphi(y) dy,
\]
for which the corresponding dispersive estimate is again valid only for small times, i.e.
\[ \|U (t,0) f \|_{L^\infty(\R^d)}\lesssim |t|^{- d/2} \|f\|_{L^1(\R^d)}, \quad \text{for $|t| < \frac{\pi}{2}$}.\]
This situation is very similar to the case of a Schr\"odinger equation with quadratic confinement \cite{Ca}. The corresponding Schr\"odinger group is formally given by
\[
S(t,0) \varphi (x)= e^{ i t (\Delta +|x|^2 )} \varphi(x) , \quad t\in \R_+,
\]
the kernel of which can be explicitly computed via Mehler's formula. Due to the existence of eigenfunctions, this time-evolution enjoys dispersive properties 
only for sufficiently small $0<|t|<\delta$, see \cite{Ca} for more details.  
\end{remark}


\section{Well-posedness results} \label{sec:well}

\subsection{The $L^2$ sub-critical case}\label{sec:sub}

As we have seen, the propagator $U(t,s)$ does not allow for uniform dispersive estimates. Nevertheless, one can prove global well-posedness of \eqref{eq:disp_nls} in the $L^2$ subcritical case.

\begin{theorem} \label{thm:L2}
Consider \eqref{eq:disp_nls} with $\varphi \in L^2(\R^d)$ and $1 < p < 1 + \frac{4}{d}$.  
Then, for any $t_0 \in \R$, there exists a unique solution $u \in C(\R,L^2(\R^d))$ satisfying 
\[ 
\| u(t,\cdot) \|_{L^2(\R^d)} = \| \varphi \|_{L^2(\R^d)}.
\]
In addition, for each admissible pair $(q, r)$ and for any compact time-interval $I\subset \R$:
\begin{equation*}
\|u\|_{L^q_tL^r_x(I\times\R^d)}\le C\big(|I|, \|u_0\|_{L^2(\R^d)}\big).
\end{equation*}
\end{theorem}
\begin{proof}
Using Duhamel's formula the solution of \eqref{eq:disp_nls} can be written as
\begin{equation}\label{eq:duhamel}
u(t,x) = U(t,t_0) \varphi(x) + i \int_{t_0}^t U(t,s) |u(s,x)|^{p-1} u(s,x) \, ds.
\end{equation}
Next, we use the decomposition \eqref{eq:decom} to split the time-interval $[t_0, t]\in \R$ into countably many sub-intervals 
\[
[t_0, t]= \bigcup_{n\in \mathbb Z} \big([t_0, t] \cap I^1_n\big) \cup \big([t_0, t] \cap I_n^2)\big),
\] 
with $I^1_n =  (n , n + t_+]$ and $I_n^2 = (n+ t_+ , n+1]$. In each of the intervals $I_n^1$, $I_n^2$ we are able to apply the Strichartz estimates stated in Lemma \ref{lem:strichartz}. 
Indeed, for $1< p<1+\frac{4}{d}$, the pair \[ (q,r) = \pare{\frac{4p}{d(p-1)}, 2p}\] is admissible and hence we have
\begin{align*}
\|u\|_{L^{\frac{4p}{d(p-1)}}_tL^{2p}_x(I\times\R^d)}\lesssim & \ \|u_0\|_{L^2}+\||u|^{p-1}u\|_{L^1_tL^2_x(I\times\R^d)}\\
\lesssim & \ \|u_0\|_{L^2}+|I|^\alpha\|u\|_{L^{\frac{4p}{d(p-1)}}_tL^{2p}_x(I\times\R^d)}^p\quad ,
\end{align*}
where $\alpha=\frac{4+d-dp}{4}>0$. Hence, if $I\subset I_n^{1,2}$ is sufficiently small, a standard continuity argument implies that there exists a solution 
$u$ to \eqref{eq:disp_nls} in $I\times\R^d$ such that
\begin{equation*}
\|u\|_{L^{\frac{4p}{d(p-1)}}_tL^{2p}_x(I\times\R^d)}\le C(\|u_0\|_{L^2(\R^d)}).
\end{equation*}
The right hand side only depends on the $L^2(\R^d)$-norm of the solution, which is uniformly bounded for all times (and in fact equal to $\| \varphi \|_{L^2}$). 
Continuity therefore implies that on each subinterval the solution exists for all times $t$ in one of the sub-interval $I_n^1$, $I_n^2$, and, in addition, we have
\[
\|u\|_{L^q_tL^r_x(I^{1,2}_n\times\R^d)}\le C(\|u_0\|_{L^2}, |I_n^{1,2}|).
\]
By considering the union of sub-intervals $I_n$ we consequently infer the existence of a solution $u\in L^\infty(\R; L^2(\R^n))$. 

In order to obtain the asserted continuity in time, we consider two different times $t_2\not = t_1\in \R_+$ for which we need to show that
\begin{equation*}
\|u(t_1, \cdot)-u(t_2, \cdot)\|_{L^2(\R^n)}\to0,\quad\textrm{as}\quad |t_1-t_2|\to0.
\end{equation*}
In view of Duhamel's formula \eqref{eq:duhamel} this requires 
\begin{equation*}
\|U(t_1,t_0) \varphi -U(t_2,t_0) \varphi \|_{L^2(\R^n)}\stackrel {t_1 \to t_2}\longrightarrow0,
\end{equation*}
which is nothing but the strong continuity of $U(t,t_0)$ stated in Lemma \ref{lem:group}. For the second term on the right hand side of \eqref{eq:duhamel} we 
write
\begin{align*}
 \int_{t_0}^{t_1} U(t_1,s) |u(s,x)|^{p-1} u(s,x) \, ds -  \int_{t_0}^{t_2} U(t_2,s) |u(s,x)|^{p-1} u(s,x) \, ds = i_1 +  i_2,
\end{align*}
where
\begin{equation*}
i_1  = \big( U(t_1,t_0) - U(t_2, t_0) \big) \int_{t_0}^{t_2} U(t_0, s) |u(s,x)|^{p-1}u(s,x) \, ds
\end{equation*}
and
\begin{equation*}
i_2  = \int_{t_2}^{t_1}U(t,s) |u(s,x)|^{p-1}u(s,x)\, ds.
\end{equation*}
By using the strong continuity of $U(t,s)$ for $i_1$, and Strichartz estimates for $i_2$ (after splitting the time-interval $[t_2, t_1]$ into unions of $I^1_n$ and $I_n^2$), 
we see that both are  $o(1)$ as $ |t_1-t_2|\to0$. Thus, the solution is continuous in time with values in $L^2(\R^d)$.
\end{proof}


\subsection{Local well-posedness in $H^1$ and finite time blow-up}\label{sec:blow}

In order to allow for $p =  1 + \frac{4}{d}$, which, in particular, includes the physically relevant situation of a 
cubic nonlinearity in $d=2$ spatial dimensions, we require $\varphi \in H^1(\R^d)$. In the following, we shall always denote by $(a)_+$ the positive part of $a\in \R$.

\begin{lemma}\label{lem:H1}
Consider \eqref{eq:disp_nls} with $\varphi \in H^1(\R^d)$ and $1<  p < 1 + \frac{4}{(d-2)_+}$. 
Then there exists a 
$T^*=T^*(\| \varphi \|_{H^1})>t_0$ and a unique 
solution $u \in C([t_0,T],H^1(\R^d))$, for all $t_0< T<T^*$. The solution is maximal in the sense that if $T^* < +\infty$ then 
\[
\lim_{t\to T^*} \| \nabla u(t, \cdot) \|_{L^2(\R^d)} = + \infty.
\]
\end{lemma}
 
Note that in dimensions $d=1,2$ we can allow for any exponent $1<p<+\infty$.

\begin{proof} Differentiating \eqref{eq:disp_nls} with respect to $x$ we see that $\nabla u$ solves the following integral equation
\[
\nabla u(t,x) = U(t,t_0) \nabla \varphi(x) + i \int_{t_0}^t U(t,s) G(u(s,x)) \, ds,
\]
where 
\[
G(u) = \frac{p+1}{2} | u |^{p-1} \nabla u + \frac{p-1}{2} | u|^{p-3} u^2 \nabla \overline{u}.
\]
The latter satisfies 
\[
G(u) \lesssim | u |^{p-1} | \nabla u|.
\]
By decomposing, as before $[t_0, t]$ into unions of $I_n^1$ and $I_n^2$ and applying Strichartz estimates together with Soblev's imbedding on each of this sub-intervals, yields the 
existence of a local in-time solution $u(t,\cdot) \in H^1(\R^d)$. The existence $T^*>0$ time thereby depends on $\| \varphi \|_{H^1}$.
\end{proof}

Clearly, if at the initial time $t_0\in \R$ it holds: $\gamma(t_0) = -\gamma_-$, i.e. if we start with a defocusing step, then we know that the solution exists at least up to $T^*\ge t_+$. 
This follows from the fact that $\gamma(t)$ is constant (and negative) for all $t\in [t_0, t_+]$ and thus the energy \eqref{eq:energy} is conserved on this time interval which in turn provides a 
uniform bound on the $H^1$ norm of $u$. 

\begin{remark} In the case of constant dispersion and if $p< 1+\frac{4}{d}$, one can infer $T^*=+\infty$, i.e. a global $H^1$ solution, using Strichartz estimates and the conservation of energy. 
Since conservation of energy does not hold in our case one has to argue differently. Namely, if $\varphi \in H^1(\R^d)$ and if $p< 1+\frac{4}{d}$, the global $L^2$ solution constructed in Theorem 
\ref{thm:L2} admits propagation of regularity and thus we indeed obtain that $u\in C(\R, H^1(\R^d))$. The proof of this result is given in \cite[Theorem 5.2.1]{Caz}.
\end{remark}

In the following we denote by $Q\in H^1(\R^d)$ the unique radially symmetric solution of the following elliptic equation:
\begin{equation}\label{eq:Q}
\Delta Q - Q + Q^{ 1 + \frac{4}{d}}=0.
\end{equation}
The existence of such ground states $Q$ has been proved in \cite{BL, BLP}. 

\begin{theorem} \label{thm:blowup}
Let $p =  1 + \frac{4}{d}$ and $\varphi \in H^1(\R^d)$ with \[\| \varphi \|_{L^2} < \gamma^{-d/4}_+  \| Q \|_{L^2}.\] 
Then $T^* = +\infty$. Moreover, this criterion is sharp.
\end{theorem}

Up to the inclusion of the scaling factor $\gamma^{-d/4}_+ $, this criterion is analogous to the case without dispersion management, cf. \cite{LP}. 
We see that the stronger the dispersion in the focusing step, the larger the mass of the initial data can be in order to infer global existence.

\begin{proof} With the local in time existence result for $H^1$ solution at hand, we need to show that $\| \nabla u(t, \cdot) \|_{L^2(\R^d)} <+ \infty$ for any finite time $t\in \R$. 
To this end, let us assume for the moment that at $t=t_0$ we start with a defocusing step, i.e. $\gamma(t_0) = - \gamma_-<0$. Denoting by $\lceil t_0 \rceil \in \N$ the ceiling of $t_0\in \R$,
we consequently have that for all $t\in [t_0, \lceil t_0 \rceil]$ the following 
energy is conserved:
\[
E_-(t):= \frac{\gamma_-}{2} \int_{\R^d} | \nabla u(t,x) |^2 dx +  \frac{d}{2d+4} \int_{\R^d}  | u(t,x)|^{2+\frac{4}{d}} dx = E_-(t_0),
\]
where we have set $p = 1+4/d$. Next, let us recall the Gagliardo-Nirenberg type inequality (see \cite{We}):
\begin{equation}\label{eq:GN}
\|f\|_{L^{2+4/n}}^{2+4/d}\le\pare{1+\frac{2}{d}}\|Q\|_{L^{2}}^{-4/d}\,
\|f\|_{L^{2}}^{4/d}\, \|\nabla f\|_{L^2}^2,
\end{equation}
where $Q$ is the solution of \eqref{eq:Q}. Using that $u(t_0)\in H^1(\R^d)$, by assumption, together with the fact that for all times $t\in \R$: $\| u(t) \|_{L^2} = \| \varphi \|_{L^2}<+ \infty$, 
the inequality \eqref{eq:GN} implies that $E_-(t)<+\infty$ for all $t\in [t_0, \lceil t_0 \rceil]$. We consequently obtain that $u(t)\in H^1(\R^d)$ exists up to $t=\lceil t_0 \rceil \in \N$ and that 
\begin{equation}\label{eq:2bounds}
\| \nabla u(\lceil t_0 \rceil, \cdot) \|^2_{L^2} \lesssim E_-(t_0), \quad \|u(\lceil t_0 \rceil, \cdot ) \|^{2+4/d}_{L^{2+4/d}} \lesssim E_-(t_0).
\end{equation}
On the time interval $[\lceil t_0 \rceil , \lceil t_0 \rceil+t_+ ]$ we are in the focusing regime. The associated energy is
\[
E_+(t):= \frac{\gamma_+}{2} \int_{\R^d} | \nabla u(t,x) |^2 dx -  \frac{d}{2d+4} \int_{\R^d}  | u(t,x)|^{2+\frac{4}{d}} dx = E_+(\lceil t_0 \rceil )<+\infty,
\]
where the last inequality follows from \eqref{eq:2bounds}.
Using again \eqref{eq:GN} we infer that for all $t\in [\lceil t_0 \rceil , \lceil t_0 \rceil+t_+ ]$ it holds:
\[
E_+(\lceil t_0 \rceil )= E_+(t)  \ge \frac{1}{2}\pare{\gamma_+-\frac{\| \varphi \|_{L^2}^{4/d}}{\|Q\|_{L^2}^{4/d}}}\|\nabla u(t , \cdot )\|_{L^2}^2.
\]
This consequently implies that if $ \|  \varphi \|_{L^2} < \gamma^{-d/4}_+  \| Q \|_{L^2}$ we have
\begin{equation*}
\|\nabla u(t)\|_{L^2}^2\le 2
\pare{\gamma_+-\frac{\|\varphi \|_{L^2}^{4/d}}{\|Q\|_{L^2}^{4/d}}}^{-1}E_+(\lceil t_0 \rceil )< +\infty, 
\end{equation*}
for all $t\in [\lceil t_0 \rceil , \lceil t_0 \rceil+t_+ ]$. Thus, we infer the existence of $u(t)\in H^1(\R^d)$ up to the time $t= \lceil t_0 \rceil+t_+$ after which we are again in the defocusing regime and the 
same argument as above applies. Clearly, this also shows that we could have started at time $t=t_0$ with a focusing step, instead of a defocusing. In both cases we obtain 
that $\| \nabla u(t) \|_{L^2(\R^d)} <+ \infty$ for any finite time $t\in \R$. 

Next, let us show that this result is indeed sharp, i.e. we need to find an initial data  $ \varphi \in H^1(\R^d)$ such that $\| \varphi \|_{L^2} = \gamma^{-d/4}_+ \| Q \|_{L^2}$ and 
the corresponding maximal time of existence for $u$ is $T^* < + \infty$. By the pseudo-conformal symmetry of the mass-critical NLS with $p =1+\frac{4}{ d}$ it is well-known that (see, e.g., \cite{Caz, LP}):
\begin{equation}\label{eq:va}
v_a(t, x)=(1-at)^{-\frac{d}{2}}Q\pare{\frac{\sqrt{\gamma_+}\, x}{1-at}}e^{ia\frac{|x|^2}{4(1-at)}}e^{i\frac{t}{1-at}}
\end{equation}
is a solution to the mass critical focusing NLS
\[
i\d_t v_a+ \gamma_+ \Delta v_a+|v_a|^{\frac{4}{d}}v_a = 0,
\]
with initial datum
\begin{equation*}
v_a(0,x) =Q(\sqrt{\gamma_+}\, x)e^{ia\frac{|x|^2}{4}},
\end{equation*}
satisfying $\| v_a(0,\cdot) \|_{L^2} = \gamma^{-d/4}_+ \| Q \|_{L^2}$. In addition, we see from \eqref{eq:va} that $v_a$ blows up at time $t=\frac{1}{a}$. 
Now if $t_0$ is such that $\gamma(t_0)= \gamma_+$, i.e. we start with a focusing step, then all we need to do is choose $a=a_*$ such that
\[ \frac{1}{a_*} < t_+ - \{ t_0 \} , \]
where $[0,1)\ni \{ t_0 \} = t_0 - \lfloor t_0 \rfloor$ is the fractional part of $t_0$. Setting $u(t_0,x) = v_{a_*}(0,x)$ ensures that the solution to \eqref{eq:disp_nls} will blow up before the time $\lfloor t_0 \rfloor + t_+$, 
i.e. before switching to the defocusing regime. If on the other hand we start in a defocusing step $\gamma(t_0) = - \gamma_-$, then we choose $u(t_0,x)  = w( t_0, x)$, where
$w$ solves the time-reversed defocusing NLS
\[
i\d_t w + \gamma_- \Delta w - |w|^{\frac{4}{d}}w = 0, 
\]
with initial data $w(\lceil t_0 \rceil, x) = v_{a_*}(0,x)$.
\end{proof}

In particular, the proof shows that a defocusing step in general can not prevent the appearance of finite time blow-up. Note, however, that the 
construction of blow-up solutions given above is based on the pseudo-conformal symmetry of the mass critical NLS and the associated 
blow-up scenario is known to be {\it unstable}, see, e.g., \cite{LP} for further discussions.

\begin{remark} 
Also note that in the proof Theorem \ref{thm:blowup} we have used the fact that $\gamma(t)$ is piecewise constant several times. 
Indeed, the situation where $\gamma(t)$ is smooth and $1-$periodic would be considerably more complicated.
\end{remark}


\section{Averaging for fast dispersion management} \label{sec:fast}

\subsection{Preliminiaries} In this section we consider the $\eps-$scaled NLS \eqref{eq:fdisp_nls}. Clearly, we have that for every {\it fixed} $\eps>0$ 
all the results of the foregoing section remain valid. In the following, though, we shall consider the 
scaling limit as $\eps \to 0_+$, i.e. the regime of rapidly varying dispersion.

As a first step, we shall show that the expected averaging result is true for the linear equation. To this end, we denote by 
$$
U_\eps (t,s) = e^{ i \Gamma_\eps (t, s) \Delta} , \quad \Gamma_\eps (t,s) = \eps \int^{t/\eps}_{s/\eps}\gamma(\tau) d\tau,
$$
the propagator associated to the fast-dispersion map $ \gamma \left(\frac{t}{\eps} \right)$. In addition,  
\[
U_0(t,s) = e^{ i \langle \gamma \rangle  (t-s) \Delta } = e^{ i t \langle \gamma \rangle   \Delta } \circ e^{- i s\langle \gamma \rangle \Delta },
\]
is the propagator associated to linear Schr\"odinger equation with averaged dispersion $\langle \gamma \rangle \in \R$ defined in \eqref{eq:average}. Note that in fact $U_0(t,s) = U_0(t-s,0)$ and thus a group.
\begin{lemma}\label{lem:linhom}
For any $s, t \in\R$ and $f \in H^\sigma(\R^d)$ with $\sigma \in \N$ we have
\[
\lim_{\eps\to0_+} \, \sup_{t,s\in \R}  \|U_\eps (t,s) f-U_0(t,s) f\|_{H^\sigma(\R^d)}=0.
\]
\end{lemma}
\begin{proof} 
Using \eqref{eq:Gamma_decom} we can decompose $\Gamma_\eps$ as
\[
\Gamma_\eps(t,s) = \langle \gamma \rangle (t-s) + \eps \int_{s/ \eps}^{t / \eps} \gamma_0(\tau) d \tau \equiv \langle \gamma \rangle (t-s) + \eps \vartheta_\eps(t,s),
\]
where we note that $\vartheta_\eps(t,s) \in L^\infty(\R^d)$ uniformly, in view of \eqref{eq:gamma_bound}. 
Fourier transformation and Plancherel's identity allow us to write
\begin{align*}
&\sup_{t,s\in \R}  \|U_\eps (t,s) f-U_0(t,s) f\|^2_{H^\sigma (\R^d)}= \\
& =\sup_{t,s\in \R} \int_{\R^d}(1+|\xi|^2)^\sigma \modu{e^{i\langle \gamma \rangle (t-s) |\xi|^2}
\pare{e^{i\eps \vartheta_\eps(t,s) }-1}}^2 |\widehat f(\xi)|^2\, d\xi.
\end{align*}
The assertion then follows by the Lebesgue dominated convergence theorem and the fact that
\[
\lim_{\eps \to 0_+} \pare{e^{i\eps \vartheta_\eps(t,s) }-1}= 0,
\]
pointwise for all $t,s \in \R$.
\end{proof}

To prove the desired averaging result as $\eps \to 0_+$, we will require sufficiently smooth solutions $u_\eps(t,\cdot) \in H^\sigma(\R^d)$ with $\sigma > \frac{d}{2}$. 
In the following we shall concentrate on the physically relevant cases of $d\le 3$, spatial dimensions. The generalization to higher order dimensions will be indicated below. 
For $d\le 3$ it is sufficient to consider solutions in $H^2(\R^d)$, whose 
existence is guaranteed by the following lemma. 

\begin{lemma}\label{lem:H2}
Let $\varphi \in H^2(\R^d)$ and $2\le p < \infty $. 
Then, for any $\eps >0$, there exists a time $T^*=T^*(\| \varphi \|_{H^2}; \eps)>t_0$ and a unique 
solution $u_\eps \in C([t_0,T],H^2(\R^d))$, for all $t_0<T<T^*$. The solution is maximal in the sense that if $T^*< +\infty$ then 
\[
\lim_{t\to T_*} \| u_\eps(t, \cdot) \|_{H^2(\R^d)} = + \infty.
\]
Furthermore 
\[
\tau:= \inf_{0< \eps \le 1} T^*(\| \varphi \|_{H^2}; \eps) > 0.
\]
\end{lemma}
\begin{proof}
The proof follows from arguments analogous to those given in the proof of 
Theorem 1.4 in Chapter 6 of \cite{Pa} (see also \cite[Section 4.8]{Caz}). In particular, this establishes the fact that $\tau >0$, bearing in mind that $\sup_{0< \eps \le 1} \| U_\eps(t,s) \| = 1$.
\end{proof}

Standard arguments also imply that for $\varphi \in H^2(\R^d)$ there exists a maximal time of existence $T_*= T_*( \| \varphi\|_{H^2})> t_0$ 
and a unique solution $u_0\in C([t_0,T], H^2 (\R^d))$, for all $t_0<T< T_*$, satisfying the averaged equation \eqref{eq:av_nls}, cf. \cite{Caz, LP} for more details. 
Depending on the power of the nonlinearity $p$ and on the sign of $\langle \gamma \rangle \in \R$ the maximal time of existence for the averaged equation 
might be infinite or not. In particular, for $H^1$-subcritical nonlinearities in the defocusing case $\langle \gamma \rangle < 0$ we have $T_* = +\infty$, cf. the results given in Section 5.3 of \cite{Caz}.

\subsection{Averaging of NLS with rapidly varying dispersion}

The main result of this section is as follows:

\begin{theorem} \label{thm:average} Let $d\le 3$, $\varphi \in H^2(\R^d)$, and $2\le p <\infty $. Denote by $u_0\in C([t_0,T_*), H^2 (\R^d))$ the 
maximal solution to the averaged equation \eqref{eq:av_nls}. In addition, let $u_\eps \in C([t_0,T^*), H^2 (\R^d))$ be the maximal solution of \eqref{eq:fdisp_nls} for given $\eps>0$.
Then, we have that
\[
 u_\eps \stackrel {\eps \to 0_+}\longrightarrow u_0 , \quad \text{in $L^\infty([t_0,T], H^2 (\R^d))$,}
\]
for all $t_0<T < T_*$. 
\end{theorem}

In particular, we know that as $\eps\to 0_+$ the solution $u_\eps$ of the original equation \eqref{eq:disp_nls} can not blow up before $T_*>0$, the maximal time of existence of a 
smooth solutions to the averaged equation.

\begin{proof} In the following, we denote the nonlinearity by $f(u) = |u|^{p-1} u $, for simplicity. In view of Duhamel's formula \eqref{eq:duhamel} we need to estimate the difference
\begin{align*}
 & u_\eps(t,x) - u_0(t,x) = (U_\eps (t,t_0) -U_0(t, t_0)) \varphi (x) \\
 & +  i \int_{t_0}^t U_\eps(t,s) f(u_\eps(s,x) )- U_0(t,s) f(u_0(s,x)) \, ds.
\end{align*}
For the first term on the right hand side we can directly use Lemma \ref{lem:linhom}, whereas the second term can be rewritten as
\begin{align*}
\int_{t_0}^t U_\eps(t,s) f(u_\eps(s,x) )- U_0(t,s) f(u_0(s,x) \, ds = \mathcal I^\eps_1 + \mathcal I^\eps_2,
\end{align*}
where
\[
\mathcal I^\eps_1  = \int_{t_0}^t U_\eps(t,s) \left( f(u_\eps(s,x)) - f( u_0(s,x)) \right) \, ds,
\]
and 
\[
\mathcal I^\eps_2 = \int_{t_0}^t \left(U_\eps(t,s)- U_0(t,s) \right) f( u_0(s,x)) \, ds.
\]
We first start by deriving an estimate in $L^2(\R^d)$, i.e. $\sigma =0$. Using Minkowski's inequality, 
\[
\| \mathcal I^\eps_2 \|_{L^\infty_t L^2_x} \le |T-t_0| \| \left(U_\eps(t,s)- U_0(t,s) \right) f( u_0(s)) \|_{L^\infty_{t,s} L^2_x}.
\]
As long as $u_0 (t,\cdot) \in H^2(\R^d) \hookrightarrow L^\infty(\R^d)$ for $d\le 3$ we have
\[
\|f( u_0)\|_{L^2} \le  \| u_0 \|^{p-1}_{L^\infty}{ \| u_0\|}_{L^2} < +\infty,
\]
and thus Lemma \ref{lem:linhom} implies that for $\eps = \eps(T)$ sufficiently small there exists a $ \delta(\eps)>0$ such that
\[
\| \mathcal I^\eps_2 \|_{L^\infty_t L^2_x} \lesssim  \delta(\eps) \stackrel {\eps \to 0_+}\longrightarrow 0.
\]
Next, we note that for all $u, v \in \mathbb C$ it holds
\[
| f(u)- f(v) | \le C ( |u|^{p-1} + |v|^{p-1}) | u - v|.
\]
With this we can estimate $\mathcal I^\eps_1$, using again Minkowski's inequality and the fact that $U_\eps$ is unitary on $L^2$, via
\begin{align*}
\| \mathcal I^\eps_2 \|_{L^\infty_t L^2_x} \lesssim \| f(u_\eps )- f( u_0)  \|_{L^1_t L^2_x} \lesssim ( \|u_\eps \|_{L^\infty_{t,x}} ^{p-1} + \|u_0 \|_{L^\infty_{t,x}} ^{p-1}) \| u_\eps - u_0  \|_{L^1_t L^2_x}
\end{align*}
As long as $u_\eps, u_0 \in H^2(\R^d) \hookrightarrow L^\infty(\R^d)$ this implies
\begin{equation} \label{eq:estimate1}
\| u_\eps  - u_0 \|_{L^\infty_t L^2_x} \le \delta(\eps) + C \| u_\eps  - u_0 \|_{L^{1}_t L^2_x}.
\end{equation}
Next, we recall the following result proved in \cite{CS} (see also \cite{CPS, PS}):
\begin{lemma}[{\cite[Lemma A.1]{CS}}]
Let $T>0$ and $1\le p < q  \le \infty$ and $A,B \ge 0$ some constants. Assume that $f\in L^q(0,T)$ satisfies, for all $t\in (0,T)$:
\[
\| f \|_{L^q(0,t)} \le A + B \| f \|_{L^p(0,t)}.
\]
Then there exists a $K = K(B, p,q, T)$ such that $\| f \|_{L^q(0,t)} \le A K$.
\end{lemma}

Using this Lemma with $q=\infty$ and $p=1$, we infer from \eqref{eq:estimate1} that
\[
\| u_\eps  - u_0\|_{L^\infty_t L^2_x} \lesssim \delta(\eps) \stackrel {\eps \to 0_+}\longrightarrow 0.
\]
To obtain the analogous estimate for  $H^\sigma$ with $\sigma =1,2$ we differentiate the equation w.r.t. $x\in \R^d$, using the fact that the nonlinearity is sufficiently smooth. Then, similar arguments as before imply that 
\begin{equation}\label{eq:estimate2}
\| u_\eps  - u_0\|_{L^\infty_t H^2_x}   \stackrel {\eps \to 0_+}\longrightarrow 0.
\end{equation}

The result then follows from a continuity argument similar to the one given in \cite{CS}: Fix $0< T < T_*$ and set $M= \sup_{0\le t \le T} \| u_0 \|_{H^2}$. 
It follows from Lemma \ref{lem:H2} that for $\|\varphi\|_{H^2} \le M$ there exists a $\tau>t_0$ such that for any $\eps >0$, $u_\eps$ exists on $[t_0, \tau]$ and
\begin{equation}\label{eq:estimate3}
\sup_{0<\eps \le 1} \| u_\eps \|_{L^\infty ((t_0,\tau), H^2)} \lesssim \| \varphi \|_{H^2}.
\end{equation}
Next, let $t_0< \ell \le T$ be such that $u_\eps$ exists on $[t_0,\ell]$ for $\eps$ sufficiently small and 
\[
\limsup_{\eps \to 0_+} \| u_\eps \|_{L^\infty((t_0, \ell), H^2)} < +\infty.
\]
Note that $\tau = \ell$ is always a possible choice. Then we deduce from \eqref{eq:estimate2} above that 
\[
\| u_\eps (\ell, \cdot) - u_0 (\ell, \cdot) \|_{H^2}  \stackrel {\eps \to 0_+}\longrightarrow 0,
\]
from which we consequently infer that 
\[
\| u_\eps(\ell, \cdot) \| _{H^2} \le M,
\]
for $\eps \ll 1$ sufficiently small. Applying Lemma \ref{lem:H2} to the NLS \eqref{eq:fdisp_nls} translated by $\ell$ we deduce that for $\eps$ sufficiently small, $u_\eps$ exists on $[t_0, \tau +\ell]$ 
and that 
\[
\sup_{0<\eps \le 1} \| u_\eps \|_{L^\infty ((t_0,\tau +\ell), H^2)} \lesssim \| \varphi \|_{H^2}.
\]
In other words, the estimate \eqref{eq:estimate3} holds with $\ell$ replaced by $\ell + \tau$, provided that $\ell + \tau \le T<T_*$. Iterating this argument we infer that 
\[
\sup_{0<\eps \le 1} \| u_\eps \|_{L^\infty ((t_0,T), H^2)} \lesssim \| \varphi \|_{H^2},
\]
and the assertion is proved.
\end{proof}

As a first corollary we can state the following result valid in arbitrary spatial dimensions $d\in \N$, provided the nonlinearity is sufficiently smooth.

\begin{corollary} 
Let $d \in \N$, $\varphi  \in H^\sigma(\R^d)$ with $\N \ni \sigma > \frac{d}{2}$ and assume that the nonlinearity satisfies $p = 1 + 2 \alpha$ with $\alpha \in \N$. Then 
\[
 u_\eps \stackrel {\eps \to 0_+}\longrightarrow u_0 , \quad \text{in $L^\infty([t_0,T], H^\sigma (\R^d))$,}
\]
for all $t_0<T < T_*$, the maximal time of existence of solutions $u_0 (t, \cdot) \in H^\sigma(\R^d)$, satisfying averaged equation.
\end{corollary}

\begin{proof} In higher dimensions, existence of smooth solutions $u_0, u_\eps(t,\cdot) \in H^\sigma(\R^d)$ with $\sigma > \frac{d}{2}$ 
can be proved as in \cite[Section 4.10]{Caz}, provided $p = 1 + 2 \alpha$ with $\alpha \in \N$. The averaging result then follows by the same arguments 
as given in the proof of Theorem \ref{thm:average} above.
\end{proof}

More interestingly, in the case of zero average dispersion $\langle \gamma \rangle = 0$ the averaged equation \eqref{eq:fdisp_nls} simplifies to an ordinary differential equation
\begin{equation}\label{eq:ode}
i\d_t u_0 + |u_0|^{p-1}u_0 = 0, \qquad u_0(t_0,x)=\varphi(x),
\end{equation}
which has been formally derived in \cite{BK2} (including higher order corrections in $\eps$). Equation \eqref{eq:ode} can be solved explicitly, resulting in the following corollary (stated for $d\le 3$, for simplicity).

\begin{corollary}\label{cor:zero} Let $d\le 3$ and 
$\langle \gamma \rangle = 0$. Under the same assumptions as in Theorem \ref{thm:average} we have, for any compact time-interval $I\subset \R$, that
\[
 u_\eps  \stackrel {\eps \to 0_+} \longrightarrow  u_0, \quad \text{in $L^\infty(I; H^2(\R^d))$,}
\]
where
\[
u_0(t,x) = \varphi(x) e^{i (t-t_0) |\varphi(x)|^{p-1}}.
\]
\end{corollary}
\begin{proof}
In the case $\langle \gamma \rangle = 0$, Theorem \ref{thm:average} implies that $ u_\eps$ converges, for $\eps \to 0_+$ to the solution of \eqref{eq:ode}. 
Multiplying the latter by $\overline u_0$ and taking the real part, we find
\[
\partial_t |u_0|^2 = 0 ,  \qquad u_0(t_0,x)=\varphi(x),
\]
and thus $| u_0 (t,x) | = |\varphi (x) |$ for all $t\in \R$. Writing $u_0 = |u_0| e^{i \theta}$ and, analogously, $\varphi = |\varphi| e^{i\theta_0}$, we find the following equation for the phase 
\[
\partial_t \theta  = |\varphi(x)|^{p-1} \theta, \quad \theta(t_0,x) = \theta_0(x).
\]
Integration w.r.t. $t$ then yields the result.
\end{proof}

The particular form of $u_0$ found in the case of zero average dispersion corresponds to a solution which {\it does not disperse}. Indeed, the spatial density (corresponding to the energy density of 
an electromagnetic pulse) is seen to be 
time {\it independent}, i.e. $|u_0(t,x)|^2 = | \varphi(x)|^2$, for all $t\in \R$. However, the solution $u_0$ oscillates with an $x$-dependent frequency $\omega(x) = |\varphi(x)|^{p-1}$. From the physics point of view, Corollary \ref{cor:zero} 
provides a possible justification for the stabilizing effect of dispersion management in optical fibers with mean zero dispersion. Note, however, that this effect should be distinguished from the ones established in \cite{ZGJT}.

\begin{remark} The case for rapidly varying mean zero dispersion $\langle \gamma \rangle = 0$ is radically different from the corresponding situation found for NLS with 
time-periodic nonlinearity management \cite{CS, DG}. In the latter case, the effective model obtained after averaging is given by a {\it linear free} Schr\"odinger equation, whose 
solution is purely dispersive, in contrast to dispersion managed NLS.
\end{remark}

\section{Concluding remarks}\label{sec:remarks}

\subsection{The case $|\gamma(t) | >0$}

In this case, the behavior of \eqref{eq:disp_nls} is purely focusing or defocusing. In particular, we have that, for any given $t_0\in \R$, the 
mapping 
\[t\mapsto \Gamma(t,t_0)\in \R\] is strictly monotone 
and we can define a new unknown
\[
v(t,x) = u(\Gamma(t,t_0), x), 
\]
which solves 
\begin{equation}\label{eq:NM_nls}
i\d_tv +\Delta v + \kappa(t) |v|^{p-1}v=0,\qquad v(0,x)=\varphi(x),
\end{equation}
with coefficient $\kappa (t) = \frac{1}{\gamma(t)}$. Equation \eqref{eq:NM_nls} is a NLS with time-dependent {\it nonlinearity management}, similar to the models studied in \cite{CS, Fa}. 
In particular, if $\gamma(t)<0$ for all $t\in \R$, this equation is defocusing and one can prove global in time existence of solutions in $H^1(\R^d)$ along the lines of \cite{Fa}.

Of course, the (physically and mathematically) most interesting case of nonlinearity management is the one with sign changing time-dependent coefficient $\kappa(t)$. Such a situation, however, 
is no longer equivalent to the one with dispersion management.

\subsection{Possible generalizations} 

Let us mention that it is straightforward to generalize all of our results to the following NLS type equation with linear dissipation
\[
i\d_tu+\gamma(t)\Delta u+|u|^{p-1}u + i \sigma u=0, \quad \sigma >0.
\]
This equation models wave propagation in dispersion managed fibers including the effects of absorption (or damping) by the fiber. In the case without dispersion management the influence 
of the damping term on the possibility of finite time blow-up is well studied, cf. \cite{Fi, OT}.

Another possibility would be to consider models with only partial dispersion management, i.e. where the dispersion management only appears in one space direction. An example of this sort 
can be found in \cite{ABS}, where the authors study the following equation (in $d=2$):
\[
i\d_tu+\big(\gamma(t)\partial_{xx} + \partial_{yy}\big) u +|u|^{2}u =0.
\]
Note that with our choice of $\gamma(t)$ this equation periodically switches in between the usual, ``elliptic" NLS and a {{\it non-elliptic}, or {\it hyperbolic}, NLS
(in the terminology of \cite{GS} and \cite{SuSu}, respectively). To this end, we recall that the Cauchy problem for the non-elliptic NLS
\[
i\d_tu- \partial_{xx} u+ \partial_{yy} u +|u|^{2}u =0
\]
is locally well-posed for initial data in $H^1(\R^2)$, see \cite{GS}. It is conjectured that this local solution is actually global.
One of the main differences between elliptic and non-elliptic NLS is the absence of 
localized solitary waves for the latter.

\bigskip
\bigskip

{\bf Acknowledgment.} The third author wants to thank I. Nenciu for helpful discussions.

\end{document}